\documentclass[leqno,draft,12pt]{amsart}
\usepackage{amssymb}
\usepackage{amsmath}
\usepackage{enumerate}
\usepackage{amsfonts}
\newtheorem{theorem}[equation]{Theorem}
\newtheorem{corollary}[equation]{Corollary}
\newtheorem{lemma}[equation]{Lemma}
\newtheorem{proposition}[equation]{Proposition}

\theoremstyle{definition}

\theoremstyle{definition}

\theoremstyle{remark}

\theoremstyle{remark}
\newtheorem*{note}{Remark}

\theoremstyle{definition}

\numberwithin{equation}{section}

\DeclareMathOperator{\re}{Re}

\newcommand{\e}{\text{\bf E}}
\newcommand{\p}{\text{\bf P}}
\newcommand{\one}{\mathbf 1}
\newcommand{\fp}{\mathfrak{p}}
\newcommand{\fP}{\mathfrak{P}}

\newcommand{\N}{\mathbb{N}}
\newcommand{\Z}{\mathbb{Z}}
\newcommand{\R}{\mathbb{R}}
\newcommand{\Q}{\mathbb{Q}}
\newcommand{\C}{\mathbb{C}}

\begin{document}

\title[Markov processes on the adeles]{Markov processes on the adeles and Dedekind's zeta function}
\subjclass[2010]{60J25, 60B05, 11R42, 11R56} \keywords{Markov process, semi-stable processes on local fields, adele ring, algebraic number fields, Dedekind's zeta function}
\begin{abstract}
Let $K$ be an algebraic number field. We construct an additive Markov process $X_t^{K_\mathbb A}$ on the ring of adeles $K_\mathbb A,$ whose coordinates $X_t^{(v)}$ are independent and use this process to give a probabilistic interpretation of the Dedekind zeta function $\zeta_K(s),$ for $\re s>1.$ This note extends a recent work of Yasuda \cite{Yprob} where the case of the field $K=\Q$ of rational numbers was considered.
\end{abstract}

\author[R. Urban]{Roman Urban}
\address{Institute of Mathematics\\
Wroclaw University\\
Plac Grunwaldzki 2/4\\
50-384 Wroclaw, Poland} \email{urban@math.uni.wroc.pl}
\thanks{Research supported in part by
the MNiSW research grant
N N201 393937.}

\maketitle

\section{Introduction}\label{introduction}
For definitions of number theoretical terms appearing in this Introduction see \S\ref{prel}.

Let $K$ be an algebraic number field (i.e., a finite extension of $\Q$).
Let $s=\sigma+it\in\C.$ The Dedekind's zeta function of an algebraic number field $K$ is defined by
\begin{equation}\label{ds}
\zeta_K(s)=\sum_I\frac{1}{N(I)^s},
\end{equation}
where the summation is over all non-zero ideals in $R_K,$ the ring of integers of $K,$ and $N(I)=[R_K:I].$ In the half-plane $\sigma>1$ the series \eqref{ds} converges absolutely, and the convergence is uniform in every compact subset of that half-plane. Moreover, by unique factorization
of ideals, we have the Euler product representation for $\sigma>1,$
\begin{equation}\label{ep}
\zeta_K(s)=\prod_\fp\left(1-\frac{1}{N(\fp)^s}\right)^{-1},
\end{equation}
where the product is over all the prime ideals $\fp\subset R_K.$

The Dedekind zeta function can be continued analytically to a meromorphic function having a unique simple pole at $s=1.$
For more about $\zeta_K$ see, e.g., \cite{Nar,Neu}.

The aim of this note is to show some probabilistic interpretation of the Euler product \eqref{ep} of the Dedekind zeta function $\zeta_K(s).$ This will require a construction of an appropriate Markov processes on the ring of adels of $K.$
Analysis of stochastic processes on local fields and adels attracted the attention of many authors (e.g., \cite{Karwowski_adeles,AK,Ytohoku}). The reason for this interest maybe that the ultrametric spaces seem to be more suitable for description of some physical phenomena. Thus, many authors study for example $p$-adic string theory or $p$-adic quantum mechanics
(see e.g., Khrennikov's monograph \cite{K} and the references therein). The ultrametric analysis has also found its application in biology and social science (see \cite{K1}).

Our main theorem of this paper is an application of stochastic analysis on ultrametric spaces to number theory. To state this result we need to introduce the ring of adeles of $K.$

Let $\mathcal P(K)$ ($\mathcal{P}_\mathrm{f}(K),$ resp.) denote the set of places (finite places, resp.) of $K.$ By $K_v$ we denote the completion of $K$ with respect to $v,$ and let $|\cdot|_v$ be the  normalized valuation (see \eqref{normval}). The adele ring of $K$ is defined as
\begin{equation*}
K_\mathbb A=\left\{x=(x_v)\in\prod_{v\in\mathcal P(K)}K_v:|x_v|_v\leq 1\text{ for all but finitely many $v\in\mathcal P_\mathrm{f}$}\right\}.
\end{equation*}
Thus, the ring of adeles of $K$ is a restricted direct product, i.e, the product $\prod_{v\in\mathcal P(K)}K_v$ relative to the \begin{equation*}
R_v=\{|x_v|_v\leq 1\},\;\;v\in\mathcal P_\mathrm{f}(K).
\end{equation*}

Our main result is the following theorem which extends the result of Yasuda \cite[Theorem~1]{Yprob}, where $K=\Q$ and the Riemann zeta function was considered.
\begin{theorem}\label{main}
Let $K$ be an algebraic number field and let, for every  $v\in\mathcal{P}_\mathrm{f}(K),$ $\{a_v(M)\}_{M\in\Z},$ be a sequence of real number satisfying:
\begin{itemize}
\item[(i)] $a_v(M+1)\leq a_v(M),$
\item[(ii)] $\lim_{M\to+\infty}a_v(M)=0,$
\item[(iii)] $0<\sum_{v\in\mathcal P(K)}a_v(0)<+\infty,$
\item[(iv)] $a_v(M)=c_vq^{-\alpha_vM}$
for some $c_v,\alpha_v>0.$
\end{itemize}
Then there exists an additive  Markov process $X_t^{K_\mathbb A}$ on the ring of adeles $K_\mathbb A,$ whose coordinates $X_t^{(v)}$ are independent, such that for every complex number $s$ with $\re s>1,$
\begin{equation*}
\zeta_K(s)=\e_0\left(\prod_{v\in\mathcal{P}_\mathrm{f}(K)}(1-q_v^{-\alpha_v})^{-1}|\pi_vX_{\tau_v}^{(v)}|_v^{-s+\alpha_v}\right),
\end{equation*}
where $\pi_v$ is a uniformizer, and
\begin{equation*}
\tau_v=\inf\{t\geq 0:X_t^{(v)}\not\in R_v\}
\end{equation*}
is the first exit time from $R_v.$
\end{theorem}
\begin{note}
The condition (iv) implies that the coordinate processes $X_t^{(v)}$ are semi-stable. We do not use this feature in this note.
For more on semi-stable process on local fields see e.g. \cite{Ytohoku}.
\end{note}
\begin{note}
We use a different construction of the Markov process $X_t^{K_\mathbb A}$ than that presented in \cite{Yprob}. As a result our proof is much shorter and seems to be easier.
\end{note}
The following functional equations for $\zeta_K$ follows easily from Theorem~\ref{main} We omit their proofs as they are identical with those given in \cite{Yprob}.
\begin{corollary}
Let $s=x+iy\in\C$ with $x>1.$ Let $\bar s$ be the complex conjugate. Then
\begin{itemize}
\item[(1)] Let $X_t^{K_\mathbb A}$ be the process corresponding to the index $\alpha_v=2x$ for all $v\in\mathcal{P}_\mathrm{f}(K).$ Then
    \begin{equation*}
    \zeta_K(s)\e_0\left(\prod_{v\in\mathcal{P}_\mathrm{f}(K)}|\pi_vX_{\tau_v}^{(v)}|_v^{ s}\right)=\zeta_K(\bar s)\e_0\left(\prod_{v\in\mathcal{P}_\mathrm{f}(K)}|\pi_vX_{\tau_v}^{(v)}|_v^{\bar s}\right).
\end{equation*}
\item[(1)] Let $X_t^{K_\mathbb A}$ be the process corresponding to the index $\alpha_v=x$ for all $v\in\mathcal{P}_\mathrm{f}(K).$ Then
    \begin{equation*}
    \zeta_K(s)=\zeta_K(x)\e_0\left(\prod_{v\in\mathcal{P}_\mathrm{f}(K)}|\pi_vX_{\tau_v}^{(v)}|_v^{-iy}\right)
\end{equation*}
and
\begin{equation*}
    \zeta_K(s)\e_0\left(\prod_{v\in\mathcal{P}_\mathrm{f}(K)}|\pi_vX_{\tau_v}^{(v)}|_v^{iy}\right)=\zeta_K(\bar s)\e_0\left(\prod_{v\in\mathcal{P}_\mathrm{f}(K)}|\pi_vX_{\tau_v}^{(v)}|_v^{-iy}\right).
\end{equation*}
\end{itemize}
\end{corollary}
\subsection*{Structure of the paper} In \S\ref{prel} we we recall some notions and elementary facts from algebraic number theory.
In particular, we define $\fp$-adic fields and, for the algebraic number field $K,$ we define the main algebraic structure of this note, the locally compact adele ring of $K.$

In \S\ref{2markov} we define an appropriate Markov process on the adeles $K_\mathbb A.$ In order to to this first we define the coordinate Markov processes on the $\fp$-adic fields.

Finally, in \S\ref{pmain} we prove Theorem~\ref{main}.
\section{Preliminaries}\label{prel}

\subsection{$\fp$-adic fields} The best references for this section are \cite{Nar,Neu}. Let $K$ be an algebraic number field (i.e., a finite extension of $\Q$). A valuation $v$ of $K$ is a homomorphism $v:K\to\R^+\cup\{0\}$ such that $v(x)=0$ if and only if $x=0,$ and and there is a real number $c\geq 1$ such that for all $x,y\in K,$ $v(xy)=v(x)v(y)$ and $v(x+y)\leq c\max\{v(x),v(y)\}.$ The absolute value $v$ is non-trivial if $v(K)\supsetneq\{0,1\}.$ The valuation $v$ is non-Archimedean if $v$ is non-trivial and we can set $c=1,$ and is said to be Archimedean otherwise.

We say that two valuation $v_1$ and $v_2$ of $K$ are equivalent if there is an $s>0$ such that $v_1(x)=v_2(x)^s$ for every $x\in K.$ An equivalence class $v$ of a non-trivial absolute value of $K$ is called a place of $K.$ A place $v$ is finite if $v$ contains a non-Archimedean absolute value, and infinite otherwise. The set of places, finite places and infinite places of $K$ is denoted by
$\mathcal P=\mathcal P(K),$ $\mathcal P_{\mathrm f}=\mathcal P_{\mathrm f}(K)$ and $\mathcal P_\infty =\mathcal P_\infty(K),$ respectively.

By Ostrovski's theorem every non-trivial valuation of $\Q$ is either equivalent to the usual absolute value $|\cdot|_\infty,$
or to the $p$-adic absolute value $|\cdot|_p$ for some rational prime $p>1,$ defined by $|0|_p=0$ and
$|p^k\frac{n}{m}|_p=p^{-k}$ for $k,n,m\in\Z$ and $p\nmid nm.$

Let $R_K$ be the ring of integers of an algebraic number field $K.$ Let $\fp$ a prime ideal of $R_K,$ $v$ the (discrete) valuation associated with $\fp$ (\cite[Theorem~3.3]{Nar}). By $K_\fp$ or $K_v$ we denote the completion of $K$ under $v,$ and we call $K_\fp$ the $\fp$-adic field. By $k$ we denote the quotient field $R_K\slash\fp,$ the residue class field. The cardinality of this residue field is a very important parameter, which we denote by $q=q_\fp=q_v.$ The extension of $v$ to $K_\fp$ will be also denoted by $v.$ The ring of integers of $K_\fp,$ $R_\fp=\{x\in K_\fp:v(x)\leq 1\}$ is the closure of the ring $R=\{x\in K:v(x)\leq 1\},$ and $\fP=\{x\in K_\fp:v(x)<1\}=\fp R_\fp$ is a prime ideal of $R_\fp,$ which is the closure of the prime ideal $\{x\in K:v(x)<1\}$ of $R.$ The invertible elements of $R_\fp$ form a group $U(R_\fp)$ of units of $K_\fp.$ The quotient fields $R_K\slash\fp$ and $R_\fp\slash\fP$ are isomorphic (\cite[Proposition~5.1]{Nar}).

We define a uniformizer for $v,$ or a local parameter, to be an element $\pi,$ also denoted by $\pi_v$ or $\pi_\fp$ of $K_\fp$ of maximal $v(\pi)$ less than $1.$ If we fix a uniformizer $\pi,$ every element of $K_\fp^*$ can be written uniquely as $x=u\pi^m$ for some $u$ with $v(u)=1$ and $m\in\Z.$ Moreover, each element $x\in K_\fp^*$ can be expressed in one and only one way as a convergent series
\begin{equation}\label{expansion}
x=\sum_{i=m}^\infty r_i\pi^i,
\end{equation}
where the coefficients $r_i$ are taken from a set $\mathcal R\subset R_\fp$ (of cardinality $q$) of representatives of the residue classes in the field
$k_\fp:=R_\fp\slash\fP$ (i.e., the canonical map $R_\fp\to k_\fp$ induces a bijection of $\mathcal R$ onto $k_\fp$).

In what follows we consider the normalized valuation
\begin{equation}\label{normval}
v(x)=q^{-k},
\end{equation}
where $k$ is the unique integer such that $x=u\pi^k$ for some unit $u.$
Let $K$ be a field with a valuation $v.$ Then $K$ is a $\fp$-adic field with the $\fp$-adic valuation if and only if $K$ is a finite extension of $\Q_p$ for a suitable $p.$ (See \cite[Theorem~5.10]{Nar}.)
\subsection{Adeles} Here we only recall a definition of the ring of adeles of $K.$ For more details see \cite{RV,W}. We write $|\cdot|_v$ for the valuation $v(\cdot).$ The set
\begin{equation*}
K_\mathbb A=\left\{x=(x_v)\in\prod_{v\in\mathcal P(K)}K_v:|x_v|_v\leq 1\text{ for all but finitely many $v\in\mathcal P_\mathrm{f}$}\right\}
\end{equation*}
furnished with the topology in which the subgroup
\begin{equation*}
\{x=(x_v)\in K_\mathbb A:|x_v|_v\leq 1\text{ for every $v\in\mathcal P_{\mathrm f}(K)$}\}\cong\prod_{v\in\mathcal P_\infty(K)}K_v\times\prod_{v\in\mathcal P_{\mathrm f}(K)}R_v,
\end{equation*}
where $R_v=\{x\in K_v:|x|_v\leq 1\},$ carries the product topology and is open in $K_\mathbb A$ is the locally compact {\em adele ring} of $K$ and its elements are called {\em adeles}.

%The Borel $\sigma$-algebra $\mathcal F_{K_\mathbb A}$ is generated by
%\begin{equation*}
%F_{K_\mathbb A}=\left\{M=\prod\right\}
%\end{equation*}

\section{Markov process on the adele ring of $K$}\label{2markov}
\subsection{Markov process on $K_\fp$}
In order to construct an appropriate process on $K_\mathbb A$ we need first to define the processes on the $\fp$-adic fields.
Rotation-invariant additive processes on $\Q_p$ were constructed by Albeverio and Karwowski in \cite{AK}. Their construction was extended to all local fields by Yasuda in \cite{Ytokyo}.

Let $K_\fp=K_v$ be a $\fp$-adic field with the $\fp$-adic valuation $|\cdot|_\fp=|\cdot|_v.$ Let $a(M)=a_\fp(M)=a_v(M),$ $M\in\Z,$ be a sequence of real number satisfying
\begin{equation}\label{s1}
a(M+1)\leq a(M)
\end{equation}
and
\begin{equation}\label{s2}
\lim_{M\to+\infty}a(M)=0.
\end{equation}
It is shown in \cite{Ytokyo} that there is one-to-one correspondence between sequences satisfying \eqref{s1} and \eqref{s2} and rotation-invariant additive Markov processes $X_t$ whose L\'{e}vy measure $\nu$ is given by $\nu(B(0,q^M)^c)=a(M),$ where $B(x,q^M)$ denotes the ball  with radius $q^M$ centered at $x,$ i.e., $\{z\in K_\fp:|x-z|_v\leq q^M\}.$
Moreover, the transition probabilities kernels $P_t^{(\fp)}=P_t^{(v)}$ of the Markov process $X_t=X_t^{(\fp)}=X_t^{(v)}$ on $K_\fp$ are given by
\begin{multline}\label{tk}
P_t^{(\fp)}(x,B(y,q^M))=\\
\begin{cases}
P_M^{(\fp)}(t)&\text{if $|x-y|_v\leq q^M,$}\\
(q-1)^{-1}q^{1-m}(P_{M+m}^{(\fp)}(t)-P_{M+m-1}^{(\fp)}(t))&\text{if $|x-y|_v=q^{M+m},$ $m\geq 1,$}
\end{cases}
\end{multline}
where
\begin{equation*}
P_M^{(\fp)}(t)=q^{-1}(q-1)\sum_{i=0}^\infty q^{-i}\exp\left(-(q-1)^{-1}(qa(M+i)-a(M+i+1))t\right)
\end{equation*}
and $a(M)$ satisfies \eqref{s1} and \eqref{s2} (the dependence on $\fp$ of the above function is through $q=q_\fp$ and $a=a_\fp$).

The following lemma follows easily from standard theory of continuous time Markov chains together with explicit formula for $P_t^{(\fp)}.$
\begin{lemma}[{\cite[Lemma~2.1]{Ytokyo}}]\label{exittime}
Let $X_t$ be the Markov process on $K_\fp$ starting from $0$ with the transition kernel $P_t$ defined in \eqref{tk}, and let, for $m\in\Z,$
\begin{equation*}
\tau^{(m)}=\inf\{t>0:|X_t|_v>q_v^m\}.
\end{equation*}
Then
\begin{equation*}
\p(\tau^{(m)}>t)=e^{-a(m)t}.
\end{equation*}
\end{lemma}
If $K=\Q$ then the prime ideals $\fp$ of $R_K=\Z$ are of the form $(p)=\{pk:k\in\Z\}.$ Then $K_\fp=\Q_p$ is the field of $p$-adic numbers, $R_\fp=\Z_p$ is the ring of $p$-adic integers. The generator $H$ of the symmetric Markov semigroup $P_t^{(p)}$ (of the process $X_t=X_t^{(p)}$ on the state space $\Q_p$) was computed in \cite[p. 15]{AK}. The computation in the general case of $K_\fp$ for algebraic number field $K$ bears no difference -- only some cosmetic changes in the notation are required -- and we get the following.
\begin{lemma}\label{genkp}
Let $P_t^{(\fp)}$ be as in \eqref{tk}, and let
\begin{equation*}
H\one_{B}(x)=\lim_{t\searrow 0}t^{-1}\left(P_t^{(\fp)}\one_{B}(x)-\one_{B}(x)\right),
\end{equation*}
where $B=B(y,q^M),$ and $\one_B$ is the indicator function of $B.$
Then,
\begin{multline*}
H\one_{B}(x)\\=
\begin{cases}
-a_v(M)&\text{if $x\in B,$}\\
q^{1-m}(q-1)^{-1}\left(a_v(M+m-1)-a_v(M+m)\right)&\text{if $d_v(x,B)=q^{M+m}.$}
\end{cases}
\end{multline*}
\end{lemma}
\subsection{Markov process on $K_\mathbb A$}\label{mpoa}
Now we are using the stochastic processes $X_t=X_t^{(\fp)}$ on $K_\fp$ constructed in the previous section to induce in a componentwise way a stochastic process on $K_\mathbb A.$ We follow \cite{Karwowski_adeles} where the construction for $\Q_\mathbb A,$ the adele ring of $\Q$ was presented (in fact there is only restricted direct product of $\prod_{p}\Q_p,$ $p$ prime number, considered in \cite{Karwowski_adeles}, not the whole adele ring).

For $v\in\mathcal P_{\infty}(K),$ we take any Markov process $X^{(v)}_t$ on $K_v$ with transition kernel $P_t^{(v)}.$
Now, for $x=(x_v)_{v\in\mathcal P(K)}$ and $B$ being the product of balls, of not necessarily equal radii, in $K_v,$ $v\in\mathcal P(K),$ define
\begin{equation*}
P_t(x,B)=\prod_{v\in\mathcal{P}_\infty(K)}P_t^{(v)}(x_v,B_v)\prod_{\fp\subset R_K}P_t^{(\fp)}(x_\fp,B_\fp).
\end{equation*}
\begin{proposition}\label{markovonadels}
Suppose that the sequences $\{a_v(M)\}_{M\in\Z},$ $v\in\mathcal P(K)$ (satisfying \eqref{s1} and \eqref{s2}) satisfy additionally
\begin{equation}\label{s3}
0<\sum_{v\in\mathcal P(K)}a_v(0)<+\infty.
\end{equation}
Then the kernel $P_t$ defined above obeys the Chapman-Kolmogorov equation. Furthermore,
for every $x\in K_\mathbb A$ and all $t\geq 0,$
\begin{equation}\label{conserv}
P_t(x,K_\mathbb A)=1.
\end{equation}
Hence, we have a Markov process $X_t^{K_\mathbb A}$ on $K_\mathbb A$ with transition kernel $P_t.$
\end{proposition}
\begin{proof}
To check the Chapman-Kolmogorov equation is an easy exercise (it goes as in \cite[p. 4649]{Karwowski_adeles}).
So we prove \eqref{conserv}. It follows from Borel-Cantelli lemma that if
\begin{equation}\label{1}
\sum_{v\in\mathcal{P}_\mathrm{f}(K)}P_t^{(v)}(x_v,R_v^c)<+\infty
\end{equation}
then
\begin{equation*}
P_t(x,K_\mathbb A^c)=\p(\text{for infinitely many $v,$ $|X_t^{(v)}|_v>1$})=0.
\end{equation*}
Hence, we need to show \eqref{1}. Since $x_v\in R_v$ for all but finitely many $v\in\mathcal{P}_\mathrm{f}(K)$ we can assume that all $x_v$ in \eqref{1} are in $R_v.$ Thus, using \eqref{tk} we have,
\begin{equation*}
\sum_{v\in\mathcal{P}_\mathrm{f}(K)}P_t^{(v)}(x_v,R_v^c)=\sum_{v\in\mathcal{P}_\mathrm{f}(K)}\left(1-P_t^{(v)}(x_v,R_v)\right)
=\sum_{v\in\mathcal{P}_\mathrm{f}(K)}\left(1-P_0^{(v)}(t)\right).
\end{equation*}
The right hand side above is equal to
\begin{equation*}
%\sum_{v\in\mathcal{P}_\mathrm{f}(K)}\left(1-q^{-1}_v(q_v-1)\sum_{i=0}^\infty %q_v^{-i}\exp\left(-(q_v-1)^{-1}(q_va(i)-a(i+1))t\right)\right)=\\
\sum_{v\in\mathcal P_\mathrm{f}}
%_{v\in\mathcal{P}_\mathrm{f}(K)}
\left(q^{-1}_v(q_v-1)\sum_{i=0}^\infty q_v^{-i}\left(1-\exp\left(-(q_v-1)^{-1}(q_va_v(i)-a_v(i+1))t\right)\right)\right).
\end{equation*}
Since for $x\geq 0,$ $1-e^{-x}\leq x,$ we can estimate the above sum by
\begin{multline*}
t\sum_{v\in\mathcal{P}_\mathrm{f}}\left(\sum_{i=0}^\infty q_v^{-i}(a_v(i)-q^{-1}_va_v(i+1))\right)
=t\sum_{v\in\mathcal{P}_\mathrm{f}}\left(a_v(0)+\lim_{i\to+\infty}\frac{a_v(i)}{q_v^i}\right).
\end{multline*}
By \eqref{s2} the right side is equal to
$t\sum_{v\in\mathcal{P}_\mathrm{f}(K)}a_v(0),$
and is finite by \eqref{s3}.
\end{proof}
\begin{note}
In order to prove Theorem~\ref{main} it is not necessary to consider the whole adele ring. Note that we do not use coordinates $K_v$ with $v$ belonging to the set $\mathcal{P}_\infty(K)$ of infinite places of $K$ in the formulation of Theorem~\ref{main}.
\end{note}
\section{Proof of Theorem~\ref{main}}\label{pmain}
Let $X_t^{K_\mathbb A}=(X_t^{(v)})_{v\in\mathcal{P}(K)}$ be the Markov process on $K_\mathbb A$ with transition kernel $P_t$ constructed in \S\ref{mpoa}. Let, for $v\in\mathcal{P}_\mathrm{f}(K),$
\begin{equation*}
\tau_v=\inf\{t\geq 0:X_t^{(v)}\not\in R_v\}.
\end{equation*}
By Lemma~\ref{exittime},
\begin{equation}\label{time}
\p_0(\tau_v>t)=e^{-a_v(0)t}.
\end{equation}
\begin{lemma}\label{mainlemma}
Let the process $X_t^{(v)}$ be defined by the sequence satisfying \eqref{s1}, \eqref{s2}, and moreover,
\begin{equation}\label{exponential}
a_v(M)=c_vq^{-\alpha_vM}
\end{equation}
for some $c_v,\alpha_v>0.$
Then
\begin{equation*}
\p_0(|X_{\tau_v}^{(v)}|_v=q^m_v)=(q_v^{\alpha_v}-1)q_v^{-m\alpha_v}.
\end{equation*}
\end{lemma}
\begin{proof}
It follows from \eqref{expansion} that for every $y\in K_\fp,$ $M\in\Z$ and $m\geq 1$ there are $(q_v-1)q_v^{m-1}$ disjoint balls $B(x,q_v^M)=\{z\in K_\fp:|x-z|_v\leq q_v^M\}$ such that $|x-y|_v=q_v^{M+m}.$ Hence, taking $M=0,$ it follows that
\begin{equation}\label{wzor1}
\p_0(|X_{\tau_v}^{(v)}|_v=q^m_v)=(q_v-1)q_v^{m-1}\p_0(X_{\tau_v}^{(v)}\in B),
\end{equation}
where $B$ is a ball of radius $1$ such that the distance\footnote{The distance between two balls $B_1$ and $B_2$ is defined as the
distance between arbitrary $x$ and $y$ such that $x\in B_1$ and $y\in B_2,$ which is $d_v(x,y)=v(x-y).$ This is well defined due to the fact that we are working in ultrametric spaces.} $d_v(R_v,B)=q_v^m.$
Hence we need to compute $\p_0(X_{\tau_v}^{(v)}\in B)$ for $d_v(R_v,B)=q_v^m.$

Let $B_0,B_1,\ldots$ be the sequence of all disjoint balls of radii $1$ in $K_\fp,$ with $B_0=B(0,1)=R_v.$ We associate with the process $X_t^{(v)}$ a continuous time Markov chain $M_t$ on the state space $\N_0=\{0,1,\ldots\}$ defined by the relation
\begin{equation}\label{relacja}
M_t=\ell\text{ if and only if }X_t^{(v)}\in B_\ell.
\end{equation}
%Let $P_{ij}(t)$ be transition kernel for $M_t.$ Clearly, if $x\in B_0=R_v,$ $B(y,1)\not=B_0$ (hence, $B(y,1)=B_j$ for some unique %$j$) and $d_v(R_v,B(y,1))=q^m,$ then
%\begin{equation}
%P_t^{(\fp)}(x,B(y,1))=P_{0j}(t)
%\end{equation}
Let $L$ be its generator. Then, for $f\in\mathcal{D}(L),$ the domain of $L,$ by \eqref{time},
\begin{equation*}
Lf(\ell)=a_v(0)\sum_{j\in\N_0}(f(j)-f(\ell))\p(X_{\tau_v}^{(v)}\in B_j\mid X_0^{(v)}\in B_\ell).
\end{equation*}
For $f=\delta_i,$ we get
\begin{multline*}
L\delta_i(\ell)=a_v(0)(1-\delta_i(\ell))\p(X_{\tau_v}^{(v)}\in B_i\mid X_0^{(v)}\in B_\ell)\\
-a_v(0)\delta_i(\ell)\sum_{j\not=i}\p(X_{\tau_v}^{(v)}\in B_j\mid X_0^{(v)}\in B_\ell).
\end{multline*}
Therefore, if $i=\ell,$
\begin{equation*}
L\delta_\ell(\ell)=-a_v(0)\sum_{j\not=\ell}\p(X_{\tau_v}^{(v)}\in B_j\mid X_0^{(v)}\in B_\ell)=-a_v(0).
\end{equation*}
And generally,
\begin{equation}\label{Ldelta}
L\delta_i(\ell)
=\begin{cases}
-a_v(0)&\text{ if $i=\ell,$}\\
a_v(0)\p(X_{\tau_v}^{(v)}\in B_i\mid X_0^{(v)}\in B_\ell)&\text{ if $i\not=\ell.$}
\end{cases}
\end{equation}
By \eqref{Ldelta} with $\ell=0,$ \eqref{relacja} Lemma~\ref{genkp} and  it follows that if $B_i\not=B_0=R_v$ and $d_v(R_v,B_i)=q^m,$ then
\begin{equation*}
\p_0(X_{\tau_v}^{(v)}\in B_i)=a_v(0)^{-1}q_v^{1-m}(q_v-1)^{-1}\left(a_v(m-1)-a_v(m)\right)
\end{equation*}
Using \eqref{exponential}
the above probability is equal to $q_v^{1-m(\alpha_v+1)}(q_v-1)^{-1}(q_v^{\alpha_v}-1).$
This together with \eqref{wzor1} finish the proof.
\end{proof}
\begin{proof}[Proof of Theorem~\ref{main}]
For $v\in\mathcal{P}_{\mathrm{f}}(K),$ we write
\begin{multline*}
\e_0\left((1-q_v^{-\alpha_v})^{-1}|\pi_vX_{\tau_v}^{(v)}|_v^{-s+\alpha_v}\right)\\
=\sum_{m=1}^\infty\e_0\left((1-q_v^{-\alpha_v})^{-1}|\pi_vX_{\tau_v}^{(v)}|_v^{-s+\alpha_v}\one_{|X_{\tau_v}^{(v)}|_v=q_v^m}\right)\\
=(1-q_v^{-\alpha_v})^{-1}q_v^{s-\alpha_v}\sum_{m=1}^\infty q_v^{m(-s+\alpha_v)}\p_0\left(|X_{\tau_v}^{(v)}|_v=q_v^m\right).
\end{multline*}
By Lemma~\ref{mainlemma}, we get
\begin{multline*}
\e_0\left((1-q_v^{-\alpha_v})^{-1}|\pi_vX_{\tau_v}^{(v)}|_v^{-s+\alpha_v}\right)\\
=(1-q_v^{-\alpha_v})^{-1}q_v^{s-\alpha_v}\sum_{m=1}^\infty q_v^{-ms}(q_v^{\alpha_v}-1)=q_v^s\sum_{m=1}^\infty q_v^{-ms}=(1-q_v^{-s})^{-1}.
\end{multline*}
Since the coordinates $X_t^{(v)},$ $v\in\mathcal{P}_{\mathrm{f}}(K),$ are independent the result follows.
\end{proof}
%\subsection*{Acknowledgements}

\end{document}